\documentclass[leqno]{siamltex704}
\usepackage{amsmath}
\usepackage{graphicx}
\usepackage{mathrsfs}
\usepackage{float}
\usepackage{amsfonts,amssymb}
\usepackage[comma, sort&compress, numbers]{natbib}
\usepackage{dsfont}
\usepackage{pifont}
\usepackage{hyperref}
\usepackage{multirow}
\usepackage{color}
\usepackage{geometry}
\allowdisplaybreaks[4]
\usepackage{geometry}
\usepackage{algorithm}

\newtheorem{remark}{Remark}[section]
\newtheorem{example}{Example}[section]

\newcommand{\bu}{{\bf u}}

\newcommand{\bw}{{\bf w}}

\newcommand{\be}{{\bf e}}
\newcommand{\bv}{{\mathbf v}}

\def\T{{\mathcal T}}

\def\pT{{\partial T}}
\def\l{{\langle}}
\def\r{{\rangle}}

\def\bbf{{\bf f}}
\def\bg{{\bf g}}

\def\bn{{\bf n}}

\def\3bar{{|\hspace{-.02in}|\hspace{-.02in}|}}

\newcommand{\bm}[1]{\mbox{\boldmath{$#1$}}}
\def\bPhi{\bm\Phi}
\def\bRT{\bm{\pi}^{RT}_{h}}

\title {An Arbitrary Order Locking-Free  Weak Galerkin Method for Linear Elasticity Problems Based on A Reconstruction Operator }

\author{
	Fuchang Huo\thanks{School of Mathematics, Jilin University, Changchun 130012, Jilin, China (huofc22@mails.jlu.edu.com)}
	\and Ruishu
	Wang\thanks{School of Mathematics, Jilin University, Changchun 130012, Jilin, China (wangrs\_math@jlu.edu.cn)} 
	\and Yanqiu Wang\thanks{School of Mathematical Sciences, Nanjing Normal University, Nanjing 210023, Jiangsu, China (yqwang@njnu.edu.cn)} 
	\and Ran Zhang\thanks{School of Mathematics, Jilin University, Changchun 130012, Jilin, China (zhangran@jlu.edu.cn)}
}
\begin{document}

\maketitle

\begin{abstract}
The weak Galerkin (WG) finite element method has shown great potential in solving various type of partial differential equations. In this paper, we propose an arbitrary order locking-free WG method for solving linear elasticity problems, with the aid of an appropriate $H(div)$-conforming displacement reconstruction operator. Optimal order locking-free error estimates in both the $H^1$-norm and the $L^2$-norm are proved, i.e., the error is independent of the $Lam\acute{e}$ constant $\lambda$. Moreover, the term $\lambda\|\nabla\cdot \mathbf{u}\|_k$ does not need to be bounded in order to achieve these estimates. We validate the accuracy and the robustness of the proposed locking-free WG algorithm by numerical experiments.
\end{abstract}

\begin{keywords} weak Galerkin finite element method, linear elasticity problem, grad-div formulation, $H(div)$-conforming displacement reconstruction, locking-free.
\end{keywords}

\begin{AMS}
Primary 65N30, 65N15, 74S05; Secondary 35J50, 74B05.
\end{AMS}

\pagestyle{myheadings}

\section{Introduction}
In this paper, we consider the  linear elasticity problems as follows: find displacement vector $\bu$ satisfying
\begin{eqnarray}\label{primal_model}
	-\mu\Delta\bu-(\lambda+\mu)\nabla(\nabla\cdot\bu)&=&\bbf,\qquad\,\text{in}\ \Omega,\\
	\bu&=&\widehat{\bg},\qquad\text{on}\ \Gamma, \label{bc1}
\end{eqnarray}
where $\Omega$ is an open bounded, connected domain in $\mathbb R^d\ (d = 2, 3)$, and the boundary $\Gamma=\partial\Omega$ is Lipschitz continuous. $\bbf$ is the body force, and $\widehat{\bg}$ is the boundary displacement function. In addition,  $\mu$ and $\lambda$ are $Lam\acute{e}$ constants, satisfying $0<\mu_1\le \mu \le \mu_2\ll\infty$ and $0<\lambda<\infty$.

The weak formulation of (\ref{primal_model})-(\ref{bc1}) can be written as: Finding $\bu\in [H^1(\Omega)]^d$ satisfying $\bu=\widehat{\bg}$ on $\partial\Omega$ and
\begin{eqnarray}\label{model-weak}
	\mu(\nabla\bu,\nabla\bv)+(\lambda+\mu)(\nabla
	\cdot\bu,\nabla \cdot\bv)&=&(\bbf,\bv), \quad  \forall \,\bv\in
	[H_0^1(\Omega)]^d,
\end{eqnarray}
where $H^1(\Omega)$ and $H_0^1(\Omega)$ are the standard Sobolev spaces defined as follows:
\begin{eqnarray*}
	H^1(\Omega)&=&\{v\in L^2(\Omega):\ \nabla v\in [L^2(\Omega)]^d\},\\
	H_0^1(\Omega)&=&\{v\in H^1(\Omega): \ v|_{\Gamma}=0\}.
\end{eqnarray*}

In elasticity theory, it is  known that the ``locking" phenomenon \cite{babuska-suri1,b2008,Ciarlet} arises when the $Lam\acute{e}$ constant $\lambda$ approaches infinity. 
Conventional finite element scheme often fails to converge to the exact solution or does not reach optimal convergence in such cases. This phenomenon is primarily attributed to the dependence of the finite element error estimates on the $Lam\acute{e}$ constant $\lambda$. Consequently, the coefficients of the error estimates  tend towards infinity when $\lambda\rightarrow \infty$, significantly impacting the computational accuracy and efficiency of the finite element scheme.
In order to overcome the locking phenomenon, some effective techniques have been proposed in various discretizations, such as, the nonconforming finite element method (NC-FEM) \cite{CR2012,LeeC,shidongyang2012}, the mixed finite element method (MFEM) \cite{Adams2005,chenHu2018,hj2015,hs2007}, the discontinuous Galerkin (DG) finite element method  \cite{Hansbo,Wihler}, the virtual element method (VEM) \cite{VirtualelasticArtioli,VirtualelasticVeiga,VirtualelasticGain}, and  the weak Galerkin (WG) finite element method  \cite{Liu,lockingw,Yisongyang}, etc.

The main purpose of this paper is to propose an arbitrary order locking-free WG method for linear elasticity problems (\ref{primal_model})-(\ref{bc1}). The WG method is an extension of the classical Galerkin finite element method. It employs weak functions and introduces  weak differential operators to replace the traditional differential operators. A stabilizer is added to ensure the weak continuity of the numerical solution. Comparing to the standard finite element method, it is usually much more convenient to design and implement high-order WG schemes.

The WG method was first proposed by J. Wang and  X. Ye for solving  second order elliptic problems \cite{wysec2}, and later applied to Navier-Stokes equations \cite{HMY2018,zzlw2018,ZhangLin2019}, Brinkman equations \cite{jlf,Mulin2022,SFLZ}, Maxwell's equations \cite{MLW14,SLM2017,wangchunmei}, biharmonic equations \cite{mwy,wwbi,ZhangZhai15}, linear elasticity \cite{Liu,lockingw,Yisongyang}, eigenvalue problems \cite{Zhaite2,Zhaite1}, interface problems \cite{wangyanqiu1,mulin2019,penghui2020}, etc.

Scholarly investigation of the WG method for linear elasticity problems can be located  in \cite{chenxie2016,Liu,Harper2020,huowangwangzhang2023,liuyujie2022,lockingw,wangzhang2018,Yisongyang,zhangyujie2015}.
Among them, methods in \cite{chenxie2016,wangzhang2018,zhangyujie2015} employ the Hellinger-Reissner mixed formulation, which is immune to the locking phenomenon. Others use the primal formulation with various techniques to avoid locking:
\cite{liuyujie2022,lockingw} present locking-free WG methods equivalent to mixed formulations, \cite{Liu, Harper2020,Yisongyang} propose  locking-free and lowest-order WG approaches using local Raviart-Thomas spaces to approximate the gradient of the displacement. In recent research,  \cite{wangwangliu2023}  constructs a penalty-free WG method on quadrilateral meshes. However, it should be noted that this method relies on utilizing high-order regularity assumptions. \cite{huowangwangzhang2023} develops a robust lowest-order WG approach that eradicates the dependence on $\lambda$ in error estimation by modifying the WG test function. This approach does not necessitate $\lambda\|\nabla\cdot\mathbf{u}\|_1$ to be bounded.

This paper uses the WG method combined with the $H(div)$-conforming displacement reconstruction technique to solve linear elasticity problems (\ref{primal_model})-(\ref{bc1}). In order to address the issue of locking, we modify the WG test functions by utilizing an $H(div)$-conforming displacement reconstruction operator. The error estimates are independent of the parameter $\lambda$, i.e., the scheme is locking-free.

The paper is structured as follows.
In Section \ref{Weak_Galerkin_Scheme}, the WG finite element scheme of linear elasticity problems (\ref{primal_model})-(\ref{bc1}) is presented.
In Section \ref{Stability_Conditions}, we study the property of the $H(div)$-conforming displacement reconstruction operator.
In Section \ref{Error_Estimates}, error equations and error estimates are established.
In Section \ref{Numerical_Results}, we present numerical results to  demonstrate the validity of theoretical analysis.

\section{The WG finite element scheme}\label{Weak_Galerkin_Scheme}
In this section, we present the WG scheme for (\ref{primal_model})-(\ref{bc1}) and study the wellposedness of the WG scheme.

Let ${\cal T}_h$ be a shape regular simplicial partition \cite{wy1202} of the domain
$\Omega\subset \mathbb{R}^d\,(d=2,3)$.
For each $T\in \T_h$, the diameter of $T$ is denoted by $h_T$. The mesh size of the partition $\T_h$ is defined by $ h=\max_{T\in \T_h} h_T$. Additionally, ${\cal E}_h$ represents the set of all edges or faces in the partition ${\cal T}_h$, while ${\cal E}^0_h$ denotes the set of all interior edges or faces in ${\cal T}_h$.
Let $k\geq1$ be an arbitrary positive integer, and $P_k(T)$ refers to the set of polynomials defined on $T$ with a maximum degree of $k$.

Define the weak finite element space $V_h$ as
\begin{align*}
	V_h=\{\bv =\{\bv_0,\bv_{b}\}: \ \bv_0|_T \in [P_k(T)]^d, \bv_{b} |_e \in [P_{k}(e)]^d, T\in \T_h, e\in{\cal E}_h\}.
\end{align*}

A subspace of $V_h$ is defined as
\begin{equation*}
	V^0_h=\{\bv =\{\bv _0,\bv_b\}\in V_h:\ \bv_b=\textbf{0} ~\text{on}~ \Gamma\}.
\end{equation*}

Define the weak differential operators and the $H(div)$-conforming displacement reconstruction operator as follows.

\begin{definition} \cite{wysec2,wy1202, wy1302}
	For $\bv =\{\bv_0,\bv_{b}\} \in V_h$, define $\nabla_{w}\bv$  and $\nabla_{w}\cdot \bv$ on each $T\in {\cal T}_h$, respectively, to be the unique matrix-valued polynomial in
	$[P_{k-1}(T)]^{d\times d}$ and the unique polynomial in $P_{k}(T)$
	satisfying
	\begin{eqnarray}\label{weakgradient}
		(\nabla_{w} \bv,\varphi)_T&=&-(\bv_0,\nabla \cdot
		\varphi)_T+\langle \bv_b,  \varphi \bn
		\rangle_{\partial T},\qquad \forall \,\varphi\in[P_{k-1}(T)]^{d\times d},\\
		(\nabla_{w}\cdot \bv,\phi)_T&=&-(\bv_0,\nabla
		\phi)_T+\langle \bv_b\cdot \bn,
		\phi\rangle_{\partial T},\qquad \forall\, \phi\in P_{k}(T),\label{weakdiv}
	\end{eqnarray}
	where $\bn$ is the unit outward normal direction on $\partial T$.
	
\end{definition}
Denote by $RT_k(T)$ the Raviart--Thomas space on $T\in {\cal T}_h$ by
$$
RT_k(T)=[P_k(T)]^d+{\bf{x}}P_k(T),\,d=2,3.
$$
\begin{definition} \cite{MuyezhangStkoes2021,wangMuStkoes2021}
	Define the  displacement reconstruction operator $\bRT: V_h\to  H(div;\Omega):=\{\bw\in [L^2(\Omega)]^d:\, \nabla\cdot \bw \in L^2(\Omega)\}$, such that $\bRT(\bv)\in RT_k(T)$ for all $\bv=\{\bv_0,\bv_b \}\in V_h$ and $T\in \T_h$ by
	\begin{eqnarray}\label{EQ:RT1}
		\int_{T} \bRT(\bv)\cdot \bw dT&=&\int_{T} \bv_0 \cdot \bw dT,\qquad  \forall \,\bw\in[P_{k-1}(T)]^{d},
		\\
		\int_{e} \bRT(\bv)\cdot \bn \phi ds&=& \int_{e} \bv_b \cdot \bn \phi ds, \qquad \forall\, \phi\in P_{k}(e), e\subset \partial T, \label{EQ:RT2}
	\end{eqnarray}
	where $\bn$ is the unit outward normal direction on $\partial T$.
\end{definition}
It is clear that one has $\nabla\cdot \bRT(\bv) \in P_k(T)$ for all $T\in\T_h$. From the property of Raviart--Thomas elements, we know that $\bRT(\bv)\cdot \bn|_e\in P_k(e)$ for all $e\in{\cal E}_h$. Hence (\ref{EQ:RT2}) immediately implies that 
\begin{equation*}
	\bRT(\bv)\cdot \bn=\bv_b\cdot \bn, \quad on\,\,each\,\, e\in{\cal E}_h.
\end{equation*}

For $\bw,\bv\in V_h$, we introduce some bilinear forms as follows:
\begin{eqnarray}\label{EQ:stabilizer}
	\mathcal{S}_h(\bw,\bv)&=&\sum_{T\in\mathcal{T}_h}h_T^{-1}\langle \bw_0-\bw_b,\bv_0-\bv_b\rangle_{\partial T},
	\\
	\mathcal{A}_h(\bw,\bv)&=&\mu\sum_{T\in\mathcal{T}_h}(\nabla_w\bw,\nabla_w\bv)_T
	+(\lambda+\mu)\sum_{T\in\mathcal{T}_h}(\nabla_w\cdot\bw,\nabla_w\cdot\bv)_T,\label{EQ:bilinearF}
	\\
	\mathcal{A}^s_h(\bw,\bv)&=&\mathcal{A}_h(\bw,\bv)+\mathcal{S}_h(\bw,\bv).\label{EQ:bilinearForm}
\end{eqnarray}

For each edge or face $e\in\mathcal{E}_h$, denote by $Q_b$  the $L^2$ orthogonal projection onto $[P_{k}(e)]^d$.

Now, we propose the following robust WG scheme.

\begin{algorithm}
	\caption{Robust WG Algorithm}
	\label{algo-primal}
	A new WG scheme for (\ref{model-weak}) is given by seeking $\bu_h=\{\bu_0, \bu_b\}\in V_h$ with $\bu_b
	= Q_b \widehat{\bg}$ on $\Gamma $ such that
	\begin{eqnarray}\label{WGA_primal}
		\mathcal{A}^s_h(\bu_h,\bv)=(\bbf,\bRT(\bv)), \qquad\forall\, \bv=\{\bv_0, \bv_b\}\in
		V_h^0.
	\end{eqnarray}
\end{algorithm}

\begin{theorem}\label{unqu}
	The WG scheme (\ref{WGA_primal}) has a unique solution.
\end{theorem}

\begin{proof}
	For a finite dimensional linear equation, we just need to prove the uniqueness of the solution, and the existence follows.
	
	Let $\bu_h^{(j)}=\{\bu_0^{(j)}, \bu_b^{(j)}\}\in V_h,\
	j=1,2$ be two solutions of (\ref{WGA_primal}), we have 
	$\bu_b^{(j)} = Q_b \widehat{\bg}$ on $ \Gamma$ and
	\begin{eqnarray*}
		\mathcal{A}^s_h(\bu_h^{(j)},\bv)=(\bbf,\bRT(\bv)), \quad\forall \,\bv=\{\bv_0,
		\bv_b\}\in V_h^0, \ j=1,2.
	\end{eqnarray*}
	Let  $\bw=\bu_h^{(1)}-\bu_h^{(2)}$ be the difference between two solutions, then we obtain that $\bw\in V_h^0$ and
	\begin{eqnarray}\label{Uniqueness}
		\mathcal{A}^s_h(\bw,\bv)=0, \qquad\forall \,\bv=\{\bv_0, \bv_b\}\in V_h^0.
	\end{eqnarray}
	By choosing $\bv=\bw$ in (\ref{Uniqueness}), we arrive at
	$$
	\mathcal{A}^s_h(\bw,\bw) = 0.
	$$
	From the definition of $\mathcal{A}^s_h(\cdot,\cdot)$, we get
	\begin{eqnarray*}
		\mu\sum_{T\in\mathcal{T}_h}(\nabla_w\bw,\nabla_w\bw)_T+(\lambda+\mu)\sum_{T\in\mathcal{T}_h}(\nabla_w\cdot\bw,\nabla_w\cdot\bw)_T\\
		+\sum_{T\in\mathcal{T}_h}h_T^{-1}\langle
		\bw_0-\bw_b,\bw_0-\bw_b\rangle_{\partial T}=0,
	\end{eqnarray*}
	which leads to
	\begin{eqnarray}
		\nabla_w\bw&=&0, \qquad \text{in}\  T,\label{EQ:weakgradto0}\\
		\nabla_w\cdot\bw&=&0, \qquad \text{in}\  T,\\
		\bw_0-\bw_b&=&0,  \qquad \text{on}\  \pT.\label{EQ:v0vbto0}
	\end{eqnarray}
	Using the definition of the discrete weak gradient (\ref{weakgradient}), we have
	\begin{eqnarray*}
		0&=&(\nabla_w\bw,\varphi)_T
		=-(\bw_{0},\nabla\cdot\varphi)_T+\langle\bw_{b},\varphi\bn\rangle_\pT
		\\
		&=&(\nabla\bw_{0},\varphi)_T-\langle
		\bw_{0}-\bw_{b},\varphi\bn\rangle_\pT,\quad \forall\,\varphi\in [P_{k-1}(T)]^{d\times d}.
	\end{eqnarray*}
	Letting $\varphi=\nabla\bw_{0}$ in the equation, together with $\bw_{0}-\bw_{b}=0$, we obtain $\nabla\bw_0=0$ on each element
	$T$. This implies that $\bw_{0}=constant$ on each element $T$. Thus, from $\bw_{0}=\bw_{b}$ on $\pT$ and
	$\bw_{b}=\textbf{0}$ on $\Gamma$, we  get $\bw_{0}\equiv0$ and $\bw_{b}\equiv0$ in $\Omega$. This completes the proof.
\end{proof}

\section{Properties of bilinear forms and  the reconstruction operator}\label{Stability_Conditions} In this section, we  investigate  some properties of bilinear forms and the $H(div)$-conforming displacement reconstruction operator.

Define
\begin{equation}\label{3barnorm}
	\3bar\bv\3bar =\left (\sum_{T\in {\cal T}_h}\|\nabla_w \bv\|_T^2+ h_T^{-1}\| \bv_0-\bv_b\|^2_{\partial
		T}\right ) ^{\frac{1}{2}},\quad \forall \,  \bv\in V_h.
\end{equation}
Similar to the proof of Theorem \ref{unqu}, it is not hard to see that
\begin{lemma}\label{lemmanorm}
	$\3bar\cdot\3bar$ is a norm in the space $V^0_h$.
\end{lemma}
Moreover, it is also obvious that
\begin{lemma}\label{lemmacoerv}
	There exists a constant  $\delta >0$, such that
	\begin{equation}\label{EQ:coercivity}
		\delta  \3bar \bv
		\3bar^2\leq \mathcal{A}^s_h(\bv,\bv), \qquad \forall\,\bv\in V_h^0.
	\end{equation}
\end{lemma}

For $T\in {\cal T}_h$, let $Q_0$ be the $L^2$ orthogonal
projection onto $[P_k(T)]^d$, ${\cal P}_h$ be the $L^2$ orthogonal
projection onto $P_{k}(T)$, and ${\bm \Pi}_h$ be the $L^2$ orthogonal
projection onto $[P_{k-1}(T )]^{d\times d}$.
Recall that for $e\in
{\cal E}_h$, $Q_b$ is the $L^2$ orthogonal projection onto $[P_{k}(e)]^d$.
Combining $Q_0$ and $Q_b$, we define $Q_h\bu=\{Q_0\bu,Q_b\bu\}\in V_h$.

\begin{lemma}\cite{wy1302} \label{ProjectionSwap} For any $\bv \in [H^1(\Omega)]^d$, we have
	\begin{align}\label{projectiondiv}
		\nabla_w \cdot(Q_h \bv)=&\,{\cal P}_h (\nabla \cdot
		\bv),  \\
		\nabla_w (Q_h \bv)=&\, {\bm \Pi}_h (\nabla \bv).\label{projectiongrad}
	\end{align}
\end{lemma}
\begin{proof}
	For all $\varphi\in P_k(T)$, from (\ref{weakdiv}) and the integration by parts, we have
	\begin{eqnarray*}
		(\nabla_w \cdot(Q_h \bv),\varphi)_T&=&-(Q_0\bv,\nabla \varphi)_T+\langle Q_b\bv \cdot \bn,\varphi\rangle_\pT \\
		&=&-(\bv,\nabla \varphi)_T+\langle \bv \cdot \bn,\varphi\rangle_\pT \\
		&=&(\nabla \cdot\bv,\varphi)_T \\
		&=&({\cal P}_h (\nabla \cdot \bv),\varphi)_T,
	\end{eqnarray*}
	which leads to  (\ref{projectiondiv}).
	
	For any $\tau\in [P_{k-1}(T)]^{d\times d}$, by using (\ref{weakgradient}) and the integration by parts, we obtain
	\begin{eqnarray*}
		(\nabla_w (Q_h \bv),\tau)_T&=&-(Q_0\bv,\nabla \cdot \tau)_T+\langle Q_b\bv,\tau \bn\rangle_\pT \\
		&=&-(\bv,\nabla \cdot \tau)_T+\langle \bv,\tau \bn\rangle_\pT \\
		&=&(\nabla\bv,\tau)_T \\
		&=&({\bm \Pi}_h (\nabla \bv),\tau)_T,
	\end{eqnarray*}
	which proves (\ref{projectiongrad}).
\end{proof}

Now, we present some results of the $H(div)$-conforming displacement reconstruction operator.
\begin{lemma}\label{RTPr}\cite{brezzi2012mixed,guermond2021finite} For any ${\bm \omega} \in RT_k(T)$, there holds
	\begin{eqnarray}\label{RT_div}
		\nabla\cdot {\bm \omega}  &\in & P_k(T),\\
		{\bm \omega} \cdot \bn &\in & P_k(\partial T), \label{RT_part}\\
		\|{\bm \omega} \|_T&\leq & \left (\|{\cal P}^{k-1}_h {\bm \omega}\|_T^2 +\sum_{e \subset  {\partial T}} h_T\| {\bm \omega}\cdot \bn\|^2_{e}\right ) ^{\frac{1}{2}}, \label{RT_pe}
	\end{eqnarray}
	where ${\cal P}^{k-1}_h$ denotes the $L^2$ projection onto $P_{k-1}(T)$, $\bn$ is the unit outward normal direction on $\partial T$.
\end{lemma}

\begin{lemma}  For all $\bv=\{\bv_{0},\bv_{b}\}\in V_h$, there holds
	\begin{equation}\label{RTProp}
		\nabla\cdot (\bRT(\bv))=\nabla_{w}\cdot\bv.
	\end{equation}
	Moreover, we have
	\begin{equation}\label{RTEs}
		\sum_{T\in{\cal	 T}_h}\|\bRT(\bv)-\bv_{0}\|^2_T \leq C\sum_{T\in{\cal T}_h}h_T\|\bv_0-\bv_b\|^2_{\partial T}\leq Ch^2\3bar \bv\3bar^2.
	\end{equation}
\end{lemma}

\begin{proof}
	From the definition of the weak divergence (\ref{weakdiv}) and the definition of the reconstruction operator $\bRT$ in (\ref{EQ:RT1})-(\ref{EQ:RT2}), we obtain
	\begin{align*}
		(\nabla\cdot (\bRT(\bv)),q)_T&=-(\bRT(\bv),\nabla q)_T + \langle \bRT(\bv) \cdot  \bn, q\rangle_{\partial T}\\
		&= -(\bv_0,\nabla q)_T + \langle \bv_{b} \cdot  \bn, q\rangle _{\partial T} \\
		&= (\nabla_w \cdot  \bv, q)_T,
	\end{align*}
	for all $q\in P_k(T)$. Using Lemma \ref{RTPr}, ones gets (\ref{RTProp}).
	
	Using (\ref{RT_pe}) and noticing that $\bv_0|_T\in P_k(T)\subset RT_k(T)$, we have
	\begin{equation}\label{RT_Estimate1}
		\begin{split}
			&\|\bv_{0}-\bRT(\bv)\|^2_T\\
			\le& \| {\cal P}^{k-1}_h(\bv_0-\bRT(\bv))\|^2_T+\sum_{e\in{\partial T}}h_T\|(\bv_0-\bRT(\bv))\cdot \bn\|^2_e.
		\end{split}
	\end{equation}
	Next, we estimate the two terms in the right-hand side of (\ref{RT_Estimate1}) one by one.
	According to definition of $\bRT$ in (\ref{EQ:RT1})-(\ref{EQ:RT2}), we have
	\begin{equation}\label{RT_Estimate11}
		\begin{split}
			&\|{\cal P}^{k-1}_h( \bv_0-\bRT(\bv))\|^2_T\\
			=&({\cal P}^{k-1}_h(\bv_0-\bRT(\bv)),{\cal P}^{k-1}_h(\bv_0-\bRT(\bv)))_T\\
			=&({\cal P}^{k-1}_h( \bv_0-\bRT(\bv)), \bv_0-\bRT(\bv))_T\\
			=&({\cal P}^{k-1}_h( \bv_0-\bRT(\bv)),\bv_0)_T-({\cal P}^{k-1}_h( \bv_0-\bRT(\bv)),\bv_0)_T\\
			=&0.
		\end{split}
	\end{equation}
	From the property of $RT_{k}(T)$ in (\ref{RT_part}) and the definition of  $\bRT$  in (\ref{EQ:RT1})-(\ref{EQ:RT2}), we get
	\begin{equation}\label{RT_Estimate111}
		\begin{split}
			&\|( \bv_0-\bRT(\bv))\cdot \bn\|^2_e\\
			=&\l(\bv_0-\bRT(\bv))\cdot \bn,( \bv_0-\bRT(\bv))\cdot \bn\r_e\\
			=&\l(\bv_0-\bv_b)\cdot \bn,( \bv_0-\bv_b)\cdot \bn\r_e\\
			=&\|(\bv_0-\bv_b)\cdot \bn\|^2_e.
		\end{split}
	\end{equation}
	Substituting (\ref{RT_Estimate11})- (\ref{RT_Estimate111}) into (\ref{RT_Estimate1}), we obtain 
	\begin{equation*}
		\sum_{T\in{\cal	 T}_h}\|\bRT(\bv)-\bv_{0}\|^2_T \leq C\sum_{T\in{\cal T}_h}h_T\|\bv_0-\bv_b\|^2_{\partial T}\leq Ch^2\3bar \bv\3bar^2.
	\end{equation*}
	This completes the proof of the lemma.	 
\end{proof}

\section{Error Estimate}\label{Error_Estimates}
In this section, we establish the error equation and study the convergence rate of the WG scheme (\ref{WGA_primal}).

Let $\bu_h=
\{\bu_0, \bu_b\} \in V_h$ be the discrete solution to the WG scheme (\ref{WGA_primal})  and $\bu$ be the exact solution to (\ref{primal_model}). Define the error function $\textbf{e}_h$ as follows:
\begin{align}\label{erfuction}
	\be_h&=\{\be_0,\be_b\}=\{Q_0\bu-\bu_0,
	Q_b\bu-\bu_b\}.
\end{align}
It is clear that $\be_h\in V_h^0$.

\subsection{Error Equation}
The goal of this section is to construct the error equation between the discrete solution $\bu_h$ and the exact solution $\bu$.

\begin{lemma} \label{ErEq}  Let $\be_h$ be the error function, there holds
	\begin{align}\label{ErrorEquation}
		\mathcal{A}^s_h(\be_h,\bv)=
		\Theta_{\bu}(\bv), \qquad \forall\, \bv\in V_h^0,
	\end{align}
	where
	\begin{align}\label{l}
		\Theta_{\bu}(\bv)&={\mathcal{G}}_\bu(\bv)-
		{\mathcal{K}}_{\bu}(\bv)+ \mathcal{S}_h(Q_h\bu,\bv),\\
		{\mathcal{G}}_\bu(\bv)&= \mu\sum_{T\in{\cal T}_h} \langle
		\bv_0-\bv_{b},(\nabla\bu-{\bm \Pi}_h\nabla\bu
		)\bn\rangle_{\partial T},\label{l1}\\
		{\mathcal{K}}_{\bu}(\bv)&=\mu\sum_{T\in{\cal T}_h}(\Delta \bu,\bv_{0}-\bRT(\bv))_{T}.\label{theta}
	\end{align}
\end{lemma}

\begin{proof}	
	According to the property of ${\bm \Pi}_h$ in (\ref{projectiongrad}), the definition of weak gradient  (\ref{weakgradient}), and the integration by parts, we obtain
	\begin{equation}\label{weakgradientQh}
		\begin{split}
			&\mu\sum_{T\in {\cal T}_h} (\nabla_w(Q_h\bu), \nabla_w\bv)_T\\
			=&\mu\sum_{T\in {\cal T}_h} ( {\bm \Pi}_h \nabla\bu,\nabla_w\bv)_T\\
			=&  -\mu\sum_{T\in {\cal T}_h}(\bv_0, \nabla\cdot ({\bm \Pi}_h \nabla\bu
			))_T+\mu\sum_{T\in {\cal T}_h}\langle \bv_{b}, ({\bm \Pi}_h \nabla
			\bu)\bn\rangle_{\partial T}\\
			=&\mu\sum_{T\in {\cal T}_h}  ( \nabla \bv_0, {\bm \Pi}_h\nabla\bu)_T-\mu\sum_{T\in {\cal T}_h}\langle
			\bv_0 - \bv_{b}, ({\bm \Pi}_h \nabla \bu)\bn\rangle_{\partial T}\\
			=& \mu\sum_{T\in {\cal T}_h} (\nabla\bv_0, \nabla\bu) _T-\mu\sum_{T\in {\cal T}_h}\langle \bv_0 - \bv_{b}
			,({\bm \Pi}_h \nabla \bu) \bn\rangle_{\partial T}.
		\end{split}
	\end{equation}
	Next, from (\ref{projectiondiv}), the property of reconstruction operator $\bRT$ in (\ref{RTProp}), $\nabla\cdot\bu$ is continuous, and  the normal component of $\bRT(\bv)$ is continuous, we can deduce the following result:
	\begin{equation}\label{EQ:weakdivQh}
		\begin{split}
			&\sum_{T\in {\cal T}_h} (\nabla_w\cdot(Q_h\bu), \nabla_w\cdot\bv)_T\\
			=&\sum_{T\in {\cal T}_h} ({\cal P}_h (\nabla\cdot\bu),\nabla_w\cdot\bv)_T\\
			=&\sum_{T\in {\cal T}_h}({\cal P}_h (\nabla\cdot\bu),\nabla\cdot\bRT(\bv))_T\\
			=&\sum_{T\in {\cal T}_h}(\nabla\cdot\bu,\nabla\cdot\bRT(\bv))_T\\
			=&\sum_{T\in {\cal T}_h}(\nabla\cdot\bu,\nabla\cdot\bRT(\bv))_T-\sum_{T\in {\cal T}_h}\langle\nabla\cdot\bu,\bRT(\bv)\cdot \bn\rangle_{\pT}\\
			=&-(\nabla(\nabla\cdot\bu),\bRT(\bv)).
		\end{split}
	\end{equation}
	
	Testing (\ref{primal_model}) by $\bRT(\bv)$ to get
	\begin{equation}\label{EQ:f_RTv}
		\begin{split}
			&(\bbf,\bRT(\bv))\\
			=&-(\mu\Delta\bu,\bRT(\bv))-(\lambda+\mu)(\nabla(\nabla\cdot\bu),\bRT(\bv))\\
			=&-(\mu\Delta\bu,\bv_{0})+(\mu\Delta\bu,\bv_{0}-\bRT(\bv))-(\lambda+\mu)(\nabla(\nabla\cdot\bu),\bRT(\bv)).
		\end{split}
	\end{equation}
	Note that 
	\begin{equation}\label{EQ:la_v0}
		-\mu(\Delta\bu,\bv_{0})
		=\mu\sum_{T\in{\cal T}_h}(\nabla\bu,\nabla\bv_{0})_T-\mu\sum_{T\in{\cal T}_h}\langle\bv_{0},(\nabla\bu)\bn\rangle_\pT.
	\end{equation}
	By substituting (\ref{weakgradientQh}), (\ref{EQ:weakdivQh}), and (\ref{EQ:la_v0}) into (\ref{EQ:f_RTv}), we have
	\begin{equation*}\label{5.7}
		\begin{split}
			&\mu \sum_{T\in{\cal T}_h} \left((\nabla_w(Q_h\bu), \nabla_w\bv)_T+
			\langle \bv_0 - \bv_{b},({\bm \Pi}_h \nabla \bu)\bn\rangle_{\partial T}-
			\langle\bv_0,(\nabla\bu) \bn \rangle_{\partial T}\right)\\
			&+(\lambda+\mu)\sum_{T\in {\cal T}_h} (\nabla_w\cdot(Q_h\bu), \nabla_w\cdot\bv)_T+(\mu\Delta\bu,\bv_{0}-\bRT(\bv))=(\bbf,\bRT(\bv)).
		\end{split}
	\end{equation*}
	With the fact that
	$
	\sum_{T\in {\cal T}_h}\langle
	\bv_b,(\nabla\bu) \bn\rangle_{\partial T}=0
	$, we get
	\begin{equation*}
		\begin{split}
			&\mu\sum_{T\in{\cal T}_h}(\nabla_w(Q_h\bu), \nabla_w\bv)_T+(\lambda+\mu)\sum_{T\in{\cal T}_h}(\nabla_w\cdot(Q_h\bu), \nabla_w\cdot\bv)_T\\
			=& (\bbf,\bRT(\bv))-\mu\sum_{T\in {\cal T}_h}(\Delta\bu,\bv_{0}-\bRT(\bv))_T\\
			&+\mu\sum_{T\in {\cal T}_h}\langle(\nabla\bu-{\bm \Pi}_h\nabla\bu)\bn,\bv_{0}-\bv_{b}\rangle_\pT,
		\end{split}
	\end{equation*}
	which implies that
	\begin{equation}\label{ereq1}
		\mathcal{A}^s_h(Q_h\textbf{u},\bv)=(\textbf{f},\bRT(\bv))+{\mathcal{G}}_\textbf{u}(\bv)-
		{\mathcal{K}}_{\bu}(\bv)+\mathcal{S}_h(Q_h\textbf{u},\bv).
	\end{equation}
	Subtracting (\ref{WGA_primal}) from (\ref{ereq1}), we  obtain (\ref{ErrorEquation}).	 
\end{proof}

\subsection{Error Estimate in the $H^1$-Norm}
In this section, we present the error estimate in the $H^1$-norm.

\begin{lemma}\label{lemmaA1} \cite{wy1202}  Let ${\cal T}_h$ be a shape regular partition  \cite{wy1202}. Let $\bw\in [H^{l+1} (\Omega)]^ d$,  $1\le l\le k$, and $0 \leq s \leq 1$. There holds
	\begin{align}\label{A1}
		\sum_{T\in{\cal
				T}_h}h_T^{2s}\|\bw-Q_0\bw\|^2_{T,s}&\leq  C
		h^{2(l+1)}\|\bw\|^2_{l+1},\\
		\sum_{T\in{\cal T}_h}h_T^{2s}\|
		\nabla\bw-{\bm \Pi}_h\nabla\bw\|^2_{T,s}&\leq
		Ch^{2l}\|\bw\|^2_{l+1},\label{A2}\\
		\sum_{T\in{\cal T}_h}h_T^{2s}\|\nabla\cdot \bw-{\cal P}_h\nabla\cdot \bw\|^2_{T,s}&\leq
		Ch^{2l}\|\nabla\cdot \bw\|^2_{l}\leq
		Ch^{2l}\|\bw\|^2_{l+1}.\label{A3}
	\end{align}
\end{lemma}

\begin{lemma}\label{LemProEstimate} For any $\bw\in [H^{l+1} (\Omega)]^d$, $1\le l\le k$, and $\bv \in V_h$, the following estimates hold true
	\begin{align}\label{A7}
		|\mathcal{S}_h(Q_h\bw,\bv)|\leq &
		Ch^l\|\bw\|_{l+1}\3bar\bv\3bar,\\
		|{\mathcal{G}}_\bw(\bv)|\leq&
		Ch^l\|\bw\|_{l+1}\3bar\bv\3bar,\label{A8}\\
		|{\mathcal{K}}_{\bw}(\bv)|\leq &
		Ch^l\|\bw\|_{l+1}\3bar\bv\3bar.\label{A9}
	\end{align}
\end{lemma}

\begin{proof}
	According to  (\ref{EQ:stabilizer}), the
	Cauchy-Schwarz inequality,  the triangle inequality, the trace inequality, and the estimate (\ref{A1}), one has
	\begin{equation*}
		\begin{split}
			|\mathcal{S}_h(Q_h\bw,\bv)|=&\left|\sum_{T\in {\cal
					T}_h}h_T^{-1}\langle Q_0\bw-Q_b\bw, \bv_0-
			\bv_{b}\rangle_{\partial
				T}\right|\\
			\leq & \left (\sum_{T\in {\cal T}_h}h_T^{-1}\|Q_0\bw-Q_b
			\bw\|_{\partial T}^2\right ) ^{\frac{1}{2}}\left (\sum_{T\in {\cal
					T}_h}h_T^{-1}\|\bv_0- \bv_{b}\|_{\partial
				T}^2\right ) ^{\frac{1}{2}}\\
			\leq & C \left (\sum_{T\in {\cal T}_h}h_T^{-1}(\|Q_0\bw-
			\bw\|_{\partial T}^2+\|Q_b\bw-
			\bw\|_{\partial T}^2)\right ) ^{\frac{1}{2}}\3bar \bv \3bar\\
			\leq & C \left (\sum_{T\in {\cal T}_h}h_T^{-1}\|Q_0\bw-
			\bw\|_{\partial T}^2\right ) ^{\frac{1}{2}}\3bar \bv \3bar\\
			\leq & C\left (\sum_{T\in {\cal T}_h}h_T^{-2}\|Q_0\bw- \bw\|_{T}^2 +
			\|\nabla(Q_0\bw- \bw)\|_{T}^2\right ) ^{\frac{1}{2}} \3bar \bv \3bar\\
			\leq & Ch^l\|\bw\|_{l+1}\3bar \bv \3bar.
		\end{split}
	\end{equation*}
	Now, we  prove (\ref{A8}). From the Cauchy-Schwarz inequality, the
	trace inequality, and the estimate (\ref{A2}), we obtain
	\begin{equation}\label{korn}
		\begin{split}
			|{\mathcal{G}}_\bw(\bv)|=&\left|\mu \sum_{T\in{\cal
					T}_h}\langle\bv_0-\bv_{b},( \nabla\bw-{\bm \Pi}_h\nabla
			\bw)\bn\rangle_{\partial T}\right|\\
			\leq & C \left (\sum_{T\in{\cal T}_h}h_T^{-1}\|
			\bv_0-\bv_{b}\|^2_{\partial T}\right)^{ \frac{1}{2}}
			\left (\sum_{T\in{\cal T}_h}h_T \|
			\nabla\bw-{\bm \Pi}_h\nabla\bw\|^2_{\partial
				T}\right ) ^{ \frac{1}{2}}\\
			\leq & C h^l\|\bw\|_{l+1}\3bar \bv\3bar.
		\end{split}
	\end{equation}
	Next, according to the definition of ${\cal P}^{k-1}_h$ and the definition of $\bRT(\bv)$ in (\ref{EQ:RT1}), we have
	\begin{equation*}
		({\cal P}^{k-1}_h(\Delta \bw),\bv_0-\bRT(\bv))_T
		=({\cal P}^{k-1}_h(\Delta \bw),\bv_0)_T-({\cal P}^{k-1}_h(\Delta \bw),\bv_0)_T
		=0.
	\end{equation*}
	Thus, it follows from the Cauchy-Schwarz inequality, and the
	estimates (\ref{RTEs}), (\ref{A3}) that
	\begin{equation*}
		\begin{split}
			|{\mathcal{K}}_{\bw}(\bv)|&=\left|\mu\sum_{T\in{\cal T}_h}(\Delta \bw,\bv_{0}-\bRT(\bv))_{T}\right|\\
			&=\left|\mu\sum_{T\in{\cal T}_h}(\Delta \bw-{\cal P}^{k-1}_h(\Delta \bw),\bv_{0}-\bRT(\bv))_{T}\right|\\
			&\leq C \left (\sum_{T\in{\cal T}_h}\|
			\Delta \bw-{\cal P}^{k-1}_h(\Delta \bw)\|^2_{T}\right ) ^{ \frac{1}{2}}
			\left (\sum_{T\in{\cal T}_h} \|  (\bv_{0}-\bRT(\bv) \|^2_{
				T}\right ) ^{ \frac{1}{2}}\\
			&\leq Ch^{l-1}\| \bw\|_{l+1}h\3bar\bv\3bar\\
			&\leq  C h^l\| \bw\|_{l+1}\3bar\bv\3bar.
		\end{split}
	\end{equation*}
	This completes the proof of the lemma.	 
\end{proof}

\begin{theorem}\label{theorem1}
	Let $\bu_h \in V_h $ be the discrete solution arising from the WG scheme (\ref{WGA_primal}) and $\bu\in
	[H^{k+1}(\Omega)]^d$ be the exact solution of
	(\ref{primal_model}). Then, there holds 
	\begin{equation}\label{th1}
		\3barQ_h\bu-\bu_h\3bar \leq
		Ch^k\|\bu\|_{k+1},
	\end{equation}
	where the constant $C>0$ is independent of $\lambda$ or the mesh size $h$.	
\end{theorem}

\begin{proof} By taking $\bv=\be_h$ in
	(\ref{ErrorEquation}), we get
	$$
	\mathcal{A}^s_h(\be_h,\be_h)=\Theta_{\textbf{u}}(\be_h).
	$$
	Using Lemma \ref{LemProEstimate} and Lemma \ref{lemmacoerv}, we obtain
	\begin{eqnarray}\label{eh}
		\delta  \3bar\be_h\3bar^2\leq \mathcal{A}^s_h(\be_h,\be_h) \leq
		Ch^k\|\textbf{u}\|_{k+1}\3bar \be_h\3bar,
	\end{eqnarray}
	which gives (\ref{th1}).
\end{proof}

According to the results presented  in Theorem \ref{theorem1}, it is evident that the displacement error, when measured in the $H^1$-norm, demonstrates independence of the $Lam\acute{e}$ constant $\lambda$. This implies that the WG Algorithm \ref{algo-primal}  is robust about the $Lam\acute{e}$ constant $\lambda$.

\subsection{ Error Estimate in the $L^2$-Norm} In this section, we establish the error estimate in the $L^2$-norm.

Consider the dual problem of seeking
$\bPhi$ satisfying
\begin{eqnarray}\label{dualityEq}
	-\mu \Delta \bPhi -(\lambda+\mu)\nabla(\nabla\cdot \bPhi)&=&\be_0,\qquad \text{in} \ \Omega,\\
	\bPhi&=&\textbf{0}, \ \qquad\text{on} \ \Gamma. \label{bdy}
\end{eqnarray}

Assume that the dual problem (\ref{dualityEq})-(\ref{bdy}) satisfies the following regularity estimate:
\begin{equation}\label{dualregularity}
	\mu\|\bPhi\|_{2}+(\lambda+\mu) \|\nabla\cdot \bPhi\|_{1}\leq C\|\be_0 \|.
\end{equation}
According to \cite{b2008,BrennerSusanne}, it is apparent that the regularity assumption is reasonable.

\begin{theorem} \label{THL2}Let $\bu_h \in V_h $ be the solution of the WG scheme  (\ref{WGA_primal})  and $\bu \in
	[H^{k+1}(\Omega)]^d$ be the exact solution of
	(\ref{primal_model}), and $\bbf \in
	[H^{k}(\Omega)]^d$. There exists a constant $C>0$ independent of $\lambda$ or the mesh size $h$, such that
	\begin{equation}\label{Q0Norm}
		\|Q_0\bu-{\bu_0}\|\leq Ch^{k+1}\big(\|\bu\|_{k+1}+\|\bbf\|_k\big).
	\end{equation}
\end{theorem}

\begin{proof} 
	Testing (\ref{dualityEq}) by $\be_0$ and using (\ref{EQ:la_v0}), the fact that $
	\sum_{T\in {\cal T}_h}\langle
	\be_b,(\nabla\bPhi) \bn\rangle_{\partial T}=0
	$, and (\ref{weakgradientQh}), we get
	\begin{align}\label{e0st1}
		\begin{split}
			\|\be_0\|^2=&-\mu\sum_{T\in {\cal T}_h}(\Delta\bPhi,\be_0)_T-(\lambda+\mu)\sum_{T\in {\cal T}_h}(\nabla(\nabla\cdot\bPhi),\be_0)_T\\
			=&\mu\sum_{T\in{\cal T}_h}(\nabla\bPhi,\nabla\be_{0})_T-\mu\sum_{T\in{\cal T}_h}\langle\be_{0}-\be_b,(\nabla\bPhi)\bn\rangle_\pT\\
			&-(\lambda+\mu)\sum_{T\in {\cal T}_h}(\nabla(\nabla\cdot\bPhi),\be_0)_T\\
			=&\mu\sum_{T\in{\cal T}_h}(\nabla_w(Q_h\bPhi), \nabla_w\textbf{e}_h)_T-(\lambda+\mu)\sum_{T\in {\cal T}_h}(\nabla(\nabla\cdot\bPhi),\be_0)_T\\
			&-\mu\sum_{T\in{\cal T}_h}\langle \be_0 - \be_b,(\nabla\bPhi-{\bm \Pi}_h\nabla\bPhi)\bn\rangle_{\partial T}\\
			=&I_1+I_2-{\mathcal{G}}_{\bPhi}(\be_h).
		\end{split}
	\end{align}
	Using the definitions of $\textbf{e}_h$ and ${\bm \Pi}_h$, (\ref{projectiongrad}), and the fact that $\sum_{T\in{\cal T}_h}(\nabla\bPhi-{\bm \Pi}_h\nabla\bPhi, {\bm \Pi}_h\nabla\bu)_T=0$, we have
	\begin{align*}
		&\sum_{T\in{\cal T}_h}(\nabla_w(Q_h\bPhi), \nabla_w\textbf{e}_h)_T\\
		=&\sum_{T\in{\cal T}_h}({\bm \Pi}_h\nabla\bPhi, \nabla_w(Q_h\bu-\bu_h))_T\\
		=&\sum_{T\in{\cal T}_h}({\bm \Pi}_h\nabla\bPhi, {\bm \Pi}_h\nabla\bu-\nabla_w\bu_h)_T\\
		=&\sum_{T\in{\cal T}_h}({\bm \Pi}_h\nabla\bPhi, \nabla\bu-\nabla_w\bu_h)_T\\
		=&\sum_{T\in{\cal T}_h}({\bm \Pi}_h\nabla\bPhi-\nabla\bPhi, \nabla\bu-\nabla_w\bu_h)_T+\sum_{T\in{\cal T}_h}(\nabla\bPhi, \nabla\bu-\nabla_w\bu_h)_T\\
		=&\sum_{T\in{\cal T}_h}({\bm \Pi}_h\nabla\bPhi-\nabla\bPhi, \nabla\bu-\nabla_w\bu_h)_T+\sum_{T\in{\cal T}_h}(\nabla\bPhi, \nabla\bu)_T\\
		&-\sum_{T\in{\cal T}_h}(\nabla_w(Q_h\bPhi), \nabla_w\bu_h)_T-\sum_{T\in{\cal T}_h}(\nabla\bPhi-{\bm \Pi}_h\nabla\bPhi, \nabla_w\bu_h)_T\\
		=&\sum_{T\in{\cal T}_h}({\bm \Pi}_h\nabla\bPhi-\nabla\bPhi, \nabla\bu)_T+\sum_{T\in{\cal T}_h}(\nabla\bPhi, \nabla\bu)_T\\
		&-\sum_{T\in{\cal T}_h}(\nabla_w(Q_h\bPhi), \nabla_w\bu_h)_T
		+\sum_{T\in{\cal T}_h}(\nabla\bPhi-{\bm \Pi}_h\nabla\bPhi, {\bm \Pi}_h\nabla\bu)_T\\
		=&\sum_{T\in{\cal T}_h}(\nabla\bPhi, \nabla\bu)_T-\sum_{T\in{\cal T}_h}(\nabla_w(Q_h\bPhi), \nabla_w\bu_h)_T\\
		&-\sum_{T\in{\cal T}_h}(\nabla\bPhi-{\bm \Pi}_h\nabla\bPhi, \nabla\bu-{\bm \Pi}_h\nabla\bu)_T.
	\end{align*}
	Thus, we arrive at
	\begin{align}\label{e0st2}
		\begin{split}
			I_1=&\mu\sum_{T\in{\cal T}_h}(\nabla_w(Q_h\bPhi), \nabla_w\textbf{e}_h)_T\\
			=&\sum_{T\in{\cal T}_h}\mu(\nabla\bPhi, \nabla\bu)_T-\sum_{T\in{\cal T}_h}\mu(\nabla_w(Q_h\bPhi), \nabla_w\bu_h)_T\\
			&-\sum_{T\in{\cal T}_h}\mu(\nabla\bPhi-{\bm \Pi}_h\nabla\bPhi, \nabla\bu-{\bm \Pi}_h\nabla\bu)_T.
		\end{split}
	\end{align}
	According to the integration by parts, the definition of ${\cal P}_h$, (\ref{weakdiv}), and the fact that $\sum_{T\in{\cal T}_h}\langle\be_{b},(\nabla\cdot\bPhi)\bn\rangle_\pT = 0$, we obtain
	\begin{align}
		&-\sum_{T\in {\cal T}_h}(\nabla(\nabla\cdot\bPhi),\be_0)_T \nonumber\\
		=&\sum_{T\in {\cal T}_h}(\nabla\cdot\bPhi,\nabla\cdot\be_0)_T-\sum_{T\in {\cal T}_h}\l(\nabla\cdot\bPhi)\bn,\be_0\r_{\partial T} \nonumber\\
		=&\sum_{T\in {\cal T}_h}({\cal P}_h\nabla\cdot\bPhi,\nabla\cdot\be_0)_T-\sum_{T\in {\cal T}_h}\l(\nabla\cdot\bPhi)\bn,\be_0\r_{\partial T}  \nonumber\\
		=&-\sum_{T\in {\cal T}_h}\l(\nabla\cdot\bPhi)\bn,\be_0\r_{\partial T}+\sum_{T\in {\cal T}_h}\l({\cal P}_h\nabla\cdot\bPhi)\bn,\be_0\r_{\partial T} \nonumber\\
		&-\sum_{T\in {\cal T}_h}(\nabla({\cal P}_h\nabla\cdot\bPhi),\be_0)_T  \nonumber\\
		=&-\sum_{T\in {\cal T}_h}\l(\nabla\cdot\bPhi)\bn,\be_0\r_{\partial T}
		+\sum_{T\in {\cal T}_h}\l({\cal P}_h\nabla\cdot\bPhi)\bn,\be_0-\be_b\r_{\partial T} \label{e0st31}\\
		&+\sum_{T\in {\cal T}_h}({\cal P}_h\nabla\cdot\bPhi,\nabla_w\cdot \be_h)_T \nonumber\\
		=&-\sum_{T\in {\cal T}_h}\l(\nabla\cdot\bPhi)\bn,\be_0-\be_b\r_{\partial T}+\sum_{T\in {\cal T}_h}\l({\cal P}_h\nabla\cdot\bPhi)\bn,\be_0-\be_b\r_{\partial T} \nonumber\\
		&+\sum_{T\in {\cal T}_h}({\cal P}_h\nabla\cdot\bPhi,\nabla_w\cdot \be_h)_T \nonumber\\
		=&-\sum_{T\in {\cal T}_h}\l(\nabla\cdot\bPhi-{\cal P}_h\nabla\cdot\bPhi)\bn,\be_0-\be_b\r_{\partial T}+\sum_{T\in {\cal T}_h}({\cal P}_h\nabla\cdot\bPhi,\nabla_w\cdot \be_h)_T \nonumber.
	\end{align}
	As for the second term $\sum_{T\in {\cal T}_h}({\cal P}_h\nabla\cdot\bPhi,\nabla_w\cdot \be_h)_T$, by using the definition of $\textbf{e}_h$, (\ref{projectiondiv}), the definition of ${\cal P}_h$, and  the fact  that $\sum_{T\in{\cal T}_h}(\nabla\cdot \bPhi-{\cal P}_h\nabla\cdot\bPhi, {\cal P}_h\nabla\cdot\bu)_T=0$, we get
	\begin{equation}\label{e0st32}
		\begin{aligned}
			&\sum_{T\in {\cal T}_h}({\cal P}_h\nabla\cdot\bPhi,\nabla_w\cdot \be_h)_T\\
			=&\sum_{T\in {\cal T}_h}({\cal P}_h\nabla\cdot\bPhi,\nabla_w\cdot (Q_h\bu-\bu_h))_T\\
			=&\sum_{T\in {\cal T}_h}({\cal P}_h\nabla\cdot\bPhi,\nabla\cdot\bu-\nabla_w\cdot\bu_h)_T\\
			=&\sum_{T\in {\cal T}_h}({\cal P}_h\nabla\cdot\bPhi-\nabla\cdot\bPhi,\nabla\cdot\bu-\nabla_w\cdot\bu_h)_T+\sum_{T\in {\cal T}_h}(\nabla\cdot\bPhi,\nabla\cdot\bu-\nabla_w\cdot\bu_h)_T\\
			=&\sum_{T\in {\cal T}_h}({\cal P}_h\nabla\cdot\bPhi-\nabla\cdot\bPhi,\nabla\cdot\bu-\nabla_w\cdot\bu_h)_T+\sum_{T\in {\cal T}_h}(\nabla\cdot\bPhi,\nabla\cdot\bu)_T\\
			&-\sum_{T\in {\cal T}_h}(\nabla_w\cdot(Q_h\bPhi),\nabla_w\cdot\bu_h)_T-\sum_{T\in {\cal T}_h}(\nabla\cdot\bPhi-{\cal P}_h\nabla \cdot\bPhi,\nabla_w\cdot\bu_h)_T\\
			=&\sum_{T\in {\cal T}_h}(\nabla\cdot\bPhi,\nabla\cdot\bu)_T-\sum_{T\in {\cal T}_h}(\nabla_w\cdot(Q_h\bPhi),\nabla_w\cdot\bu_h)_T\\
			&-\sum_{T\in {\cal T}_h}(\nabla\cdot\bPhi-{\cal P}_h\nabla \cdot\bPhi,\nabla\cdot\bu-{\cal P}_h\nabla \cdot\bu)_T.
		\end{aligned}
	\end{equation}
	Combining (\ref{e0st31}) and (\ref{e0st32}), we obtain
	\begin{align}
		I_2=&-(\lambda+\mu)\sum_{T\in {\cal T}_h}(\nabla(\nabla\cdot\bPhi),\be_0)_T \nonumber\\
		=&-(\lambda+\mu)\sum_{T\in {\cal T}_h}\l(\nabla\cdot\bPhi-{\cal P}_h\nabla\cdot\bPhi)\bn,\be_0-\be_b\r_{\partial T} \nonumber\\
		&+(\lambda+\mu)\sum_{T\in {\cal T}_h}({\cal P}_h\nabla\cdot\bPhi,\nabla_w\cdot \be_h)_T \nonumber\\
		=&-(\lambda+\mu)\sum_{T\in {\cal T}_h}\l(\nabla\cdot\bPhi-{\cal P}_h\nabla\cdot\bPhi)\bn,\be_0-\be_b\r_{\partial T} \label{e0st33}\\
		&+(\lambda+\mu)\sum_{T\in {\cal T}_h}(\nabla\cdot\bPhi,\nabla\cdot\bu)_T \nonumber\\
		&-(\lambda+\mu)\sum_{T\in {\cal T}_h}(\nabla_w\cdot(Q_h\bPhi),\nabla_w\cdot\bu_h)_T \nonumber\\
		&-(\lambda+\mu)\sum_{T\in {\cal T}_h}(\nabla\cdot\bPhi-{\cal P}_h\nabla \cdot\bPhi,\nabla\cdot\bu-{\cal P}_h\nabla \cdot\bu)_T \nonumber.
	\end{align}
	Testing (\ref{primal_model}) by using $\bPhi$, with the boundary condition (\ref{bdy}), we have
	\begin{equation}\label{e0st41}
		\mu(\nabla\bu,\nabla\bPhi)+(\lambda+\mu)(\nabla\cdot\bu,\nabla\cdot\bPhi)=(\bbf,\bPhi).
	\end{equation}
	According to the WG scheme (\ref{WGA_primal}), we get
	\begin{equation}\label{e0st42}
		\begin{aligned}
			&\mu\sum_{T\in {\cal T}_h}(\nabla_w\bu_h,\nabla_w(Q_h\bPhi))_T+(\lambda+\mu)\sum_{T\in {\cal T}_h}(\nabla_w\cdot\bu_h,\nabla_w\cdot(Q_h\bPhi))_T\\
			=&(\bbf,\bRT(Q_h\bPhi))-\mathcal{S}_h(\bu_h,Q_h\bPhi)\\
			=&(\bbf,\bRT(Q_h\bPhi))-\mathcal{S}_h(Q_h\bu,Q_h\bPhi)+\mathcal{S}_h(\be_h,Q_h\bPhi).
		\end{aligned}
	\end{equation}
	Furthermore, combining (\ref{e0st1}), (\ref{e0st2}), (\ref{e0st33}), (\ref{e0st41}), and (\ref{e0st42}), we arrive at
	\begin{equation}\label{e0st51}
		\begin{aligned}
			\|\be_0\|^2
			=&-\mu\sum_{T\in{\cal T}_h}\langle \be_0 - \be_b,(\nabla\bPhi-{\bm \Pi}_h\nabla\bPhi)\bn\rangle_{\partial T}\\
			&-(\lambda+\mu)\sum_{T\in {\cal T}_h}\l(\nabla\cdot\bPhi-{\cal P}_h\nabla\cdot\bPhi)\bn,\be_0-\be_b\r_{\partial T}\\
			&-\mu\sum_{T\in{\cal T}_h}(\nabla\bPhi-{\bm \Pi}_h\nabla\bPhi, \nabla\bu-{\bm \Pi}_h\nabla\bu)_T\\
			&-(\lambda+\mu)\sum_{T\in {\cal T}_h}(\nabla\cdot\bPhi-{\cal P}_h\nabla \cdot\bPhi,\nabla\cdot\bu-{\cal P}_h\nabla \cdot\bu)_T\\
			&+(\bbf,\bPhi-\bRT(Q_h\bPhi))+\mathcal{S}_h(Q_h\bu,Q_h\bPhi)-\mathcal{S}_h(\be_h,Q_h\bPhi).
		\end{aligned}	
	\end{equation}
	Next, we estimate each item in  the right-hand side of (\ref{e0st51})  one by one.\\
	(i) It follows from Lemma \ref{LemProEstimate} and Theorem \ref{theorem1} that 
	\begin{equation}\label{e0st52}
		\left|\mu\sum_{T\in{\cal T}_h}\langle \be_0 - \be_b,(\nabla\bPhi-{\bm \Pi}_h\nabla\bPhi) \bn\rangle_{\partial T}\right|\leq Ch\|\bPhi\|_2\3bar\be_h\3bar\leq Ch^{k+1}\|\bPhi\|_2\|\bu\|_{k+1},
	\end{equation}
	\begin{equation}\label{e0st53}
		\left|\mathcal{S}_h(\be_h,Q_h\bPhi)\right|\leq Ch\|\bPhi\|_2\3bar\be_h\3bar\leq Ch^{k+1}\|\bPhi\|_2\|\bu\|_{k+1}.
	\end{equation}
	(ii) Using the Cauchy-Schwarz inequality, the trace inequality, the estimate (\ref{A3}), and Theorem \ref{theorem1}, we arrive at
	\begin{equation}\label{e0st54}
		\begin{split}
			&\left|\sum_{T\in{\cal T}_h}(\lambda+\mu) \langle
			\nabla\cdot\bPhi- {\cal P}_h(\nabla\cdot\bPhi) ,
			({\textbf{e}_0-\textbf{e}_b})\cdot\bn \rangle_{\partial T}  \right|\\
			&\leq C(\lambda+\mu) \left (\sum_{T\in{\cal T}_h}h_T \| \nabla\cdot\bPhi- {\cal P}_h(\nabla\cdot\bPhi) \|^2_{\partial
				T}\right ) ^{ \frac{1}{2}} \left (\sum_{T\in{\cal T}_h}h_T^{-1}\|
			(\textbf{e}_0-\textbf{e}_b)\cdot\bn\|^2_{\partial T}\right ) ^{ \frac{1}{2}}
			\\
			&\leq  C (\lambda+\mu) h\| \nabla\cdot\bPhi\|_1\3bar\textbf{e}_h\3bar\\
			&\leq  C (\lambda+\mu) h^{k+1}\| \nabla\cdot\bPhi\|_1\|\bu\|_{k+1}.
		\end{split}
	\end{equation}
	From the Cauchy-Schwarz inequality, the triangle inequality, the trace inequality, and the estimate (\ref{A1}), we obtain
	\begin{equation}\label{e0st55}
		\begin{split}
			|\mathcal{S}_h(Q_h\bu, Q_h\bPhi)|=&\left|\sum_{T\in {\cal T}_h}h_T^{-1}\langle
			Q_0\bu-Q_b\bu,Q_0\bPhi-Q_b\bPhi\rangle_{\partial
				T}\right|\\
			\leq & C\left (\sum_{T\in {\cal T}_h}
			h_T^{-1}(\|Q_0\bu-\bu\|^2_{\partial
				T}+\|Q_b\bu-\bu\|^2_{\partial
				T})\right ) ^{\frac{1}{2}}\\
			&\left (\sum_{T\in {\cal T}_h}
			h_T^{-1}(\|Q_0\bPhi-\bPhi\|^2_{\partial T}+\|Q_b\bPhi-\bPhi\|^2_{\partial T})\right ) ^{\frac{1}{2}}\\
			\leq & C\left (\sum_{T\in {\cal T}_h}
			h_T^{-1}\|Q_0\bu-\bu\|^2_{\partial
				T}\right ) ^{\frac{1}{2}}\left (\sum_{T\in {\cal T}_h}
			h_T^{-1}\|Q_0\bPhi-\bPhi\|^2_{\partial T}\right ) ^{\frac{1}{2}}\\
			\leq & Ch^{k+1}\|\bu\|_{k+1} \|\bPhi\|_{2}.
		\end{split}
	\end{equation}
	(iii) By the Cauchy-Schwarz inequality and the estimate (\ref{A2}), it follows that
	\begin{equation}\label{e0st56}
		\begin{split}
			&\left|\mu \sum_{T\in{\cal T}_h}(\nabla\bPhi-{\bm \Pi}_h\nabla\bPhi, \nabla\bu-{\bm \Pi}_h\nabla\bu)_T \right|\\
			&\leq C \left (\sum_{T\in{\cal T}_h} \| \nabla\bPhi-{\bm \Pi}_h\nabla\bPhi \|^2_{
				T}\right ) ^{ \frac{1}{2}} \left (\sum_{T\in{\cal T}_h} \| \nabla\bu-{\bm \Pi}_h\nabla\bu \|^2_{
				T}\right ) ^{ \frac{1}{2}} \\
			&\leq Ch^{k+1}\|\bu\|_{k+1}\|\bPhi\|_2.
		\end{split}
	\end{equation}
	Applying the Cauchy-Schwarz inequality  and the estimate (\ref{A3}), we get
	\begin{equation}\label{e0st57}
		\begin{split}
			&\left| (\lambda+\mu)\sum_{T\in {\cal T}_h}(\nabla\cdot\bPhi-{\cal P}_h\nabla \cdot\bPhi,\nabla\cdot\bu-{\cal P}_h\nabla \cdot\bu)_T \right|\\
			&\leq C (\lambda+\mu)\left (\sum_{T\in{\cal T}_h} \| \nabla\cdot\bPhi-{\cal P}_h\nabla \cdot\bPhi \|^2_{
				T}\right ) ^{ \frac{1}{2}} \left (\sum_{T\in{\cal T}_h} \| \nabla\cdot\bu-{\cal P}_h\nabla \cdot\bu \|^2_{
				T}\right ) ^{ \frac{1}{2}} \\
			&\leq C(\lambda+\mu)h^{k+1}\|\bu\|_{k+1}\|\nabla\cdot\bPhi\|_1.
		\end{split}
	\end{equation}
	(iv) We then estimate term $(\bbf,\bPhi-\bRT(Q_h\bPhi))$. Let $Q_0^{k-1}$ be the $L^2$ 
	orthogonal projection onto $[P_{k-1}(T)]^d$, by using the definition of $\bRT$ in (\ref{EQ:RT1}), we get 
	\begin{equation*}
		\sum_{T\in{\cal T}_h}(Q_0^{k-1}\bbf,\bPhi-\bRT(Q_h\bPhi))_T=\sum_{T\in{\cal T}_h}(Q_0^{k-1}\bbf,\bPhi-Q_0\bPhi)_T=0.
	\end{equation*}
	Furthermore, by the Cauchy-Schwarz inequality, the estimates (\ref{A1}) and (\ref{RTEs}), we obtain
	\begin{equation}\label{e0st58}
		\begin{split}
			&\left|(\bbf,\bPhi-\bRT(Q_h\bPhi))\right|\\
			&=\left|\sum_{T\in{\cal T}_h}(\bbf-Q_0^{k-1}\bbf,\bPhi-\bRT(Q_h\bPhi))\right|\\
			&= \left (\sum_{T\in{\cal T}_h}\|\bbf-Q_0^{k-1}\bbf\|^2_T\right ) ^{ \frac{1}{2}} \left (\sum_{T\in{\cal T}_h}\|\bPhi-\bRT(Q_h\bPhi)\|^2_T\right ) ^{ \frac{1}{2}}\\
			&\leq Ch^{k}\|\bbf\|_{k}h\|\bPhi\|_1\\
			&\leq Ch^{k+1}\|\bbf\|_{k}\|\bPhi\|_2.
		\end{split}
	\end{equation}
	Combining the estimates (\ref{e0st52})-(\ref{e0st58}) in (i)-(iv) and the regularity assumption (\ref{dualregularity}), we arrive at
	\begin{equation*}
		\begin{split}
			\|\be_0\|^2&\leq Ch^{k+1}\|\bPhi\|_2\|\bu\|_{k+1}+C(\lambda+\mu)h^{k+1}\|\nabla\cdot \bPhi\|_1\|\bu\|_{k+1}+Ch^{k+1}\|\bbf\|_k\|\bPhi\|_2\\
			&\leq Ch^{k+1}(\|\bPhi\|_2+(\lambda+\mu)\|\nabla\cdot \bPhi\|_1)\|\bu\|_{k+1}+Ch^{k+1}\|\bbf\|_k\|\bPhi\|_2\\
			&\leq Ch^{k+1}(\|\bu\|_{k+1}+\|\bbf\|_k)\|\be_0\|,
		\end{split}
	\end{equation*}
	which  yields the error estimate (\ref{Q0Norm}). This completes the proof of the theorem.
\end{proof}

\begin{remark}
	Based on the proofs of Theorem \ref{theorem1} and Theorem \ref{THL2}, it becomes apparent that the introduction of an $H(div)$-conforming displacement reconstruction operator in our scheme effectively eradicates the dependence of displacement error on $\lambda$. Moreover, it is essential to highlight that our method does not require the imposition of high-order regularity assumptions, and it exhibits robustness even in scenarios where $\lambda\|\nabla\cdot\bu\|_k$ is unbounded. 
\end{remark}

\section{Numerical Results}\label{Numerical_Results}
This section provides numerical examples to validate the theoretical conclusions on the WG scheme (\ref{WGA_primal}) for the elasticity problems. We use triangular meshes as shown in Figure \ref{2Dlevel}.

\begin{example}\label{example4.1}(Convergence test)
	Consider the elasticity problems (\ref{primal_model})-(\ref{bc1}) in square domain $\Omega=(0,1)^2$. The exact solution $\bu$ is chosen as follows
	\begin{eqnarray*}
		\bu=\left(\begin{array}{ccc}
			\sin(\pi x) \cos(\pi y) \\
			\cos(\pi x) \sin(\pi y)
		\end{array}\right).
	\end{eqnarray*}
	Set  $Lam\acute{e}$ constants $\mu=1$ and $\lambda=1$.
\end{example}

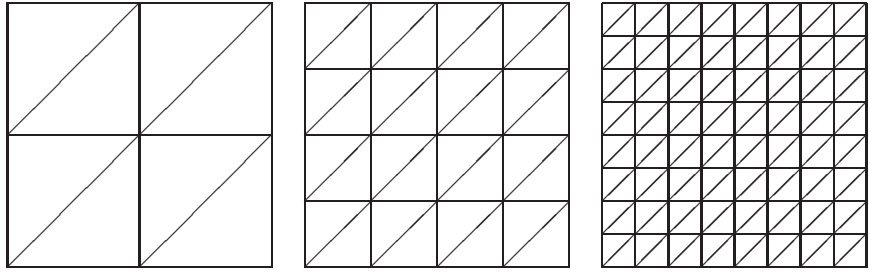
\begin{figure}[h!]
	\begin{center} \setlength\unitlength{1.25 pt}
		\begin{picture}(260,80)(0,0)
			\def\tr{\begin{picture}(20,20)(0,0)\put(0,0){\line(1,0){20}}\put(0,20){\line(1,0){20}}
					\put(0,0){\line(0,1){20}} \put(20,0){\line(0,1){20}} \put(0,0){\line(1,1){20}}
			\end{picture}}
			{\setlength\unitlength{2.5 pt}
				\multiput(0,0)(20,0){2}{\multiput(0,0)(0,20){2}{\tr}}}
			
			{\setlength\unitlength{1.25 pt}
				\multiput(90,0)(20,0){4}{\multiput(0,0)(0,20){4}{\tr}}}
			{\setlength\unitlength{0.625 pt}
				\multiput(360,0)(20,0){8}{\multiput(0,0)(0,20){8}{\tr}}}
	\end{picture}\end{center}
	\caption{The first three levels of triangular grids in square domain}
	\label{2Dlevel}
\end{figure}

The results presented in Table \ref{tab:example4.1} demonstrate the evaluation of error and numerical convergence order. Specifically, Table \ref{tab:example4.1} shows that the approximate displacement $\bu_h$ converges towards the exact solution $\bu$, indicating the effectiveness of the WG method. Additionally, it is observed that the numerical convergence order in $L^2$-norm for the $P_1-P_1$ WG elements is $O(h^2)$, while for the $P_2-P_2$ WG elements, it is $O(h^3)$. This is consistent with theoretical analysis.

\begin{table}[ht]
	\centering
	\caption{Error and convergence order of displacement $\bu$ in Example \ref{example4.1}}
	\label{tab:example4.1}
	\begin{tabular}{ccccc}
		\hline
		Level&$\3barQ_h\bu-\bu_h\3bar$&order&$\|Q_0\bu-\bu_0\|$&order\\
		\hline
		\multicolumn{5}{c}{by the $P_1-P_1$ WG elements} \\
		\hline
		2     & 4.7978e-01 &   --  & 3.9675e-02 & --\\
		3     & 2.7672e-01 & 0.7939 & 1.1751e-02 & 1.7554  \\
		4     & 1.4363e-01 & 0.9461 & 3.0846e-03 & 1.9297  \\
		5     & 7.2500e-02 & 0.9863 & 7.8110e-04 & 1.9815  \\
		6     & 3.6337e-02 & 0.9965 & 1.9592e-04 & 1.9953  \\
		7     & 1.8179e-02 & 0.9992 & 4.9020e-05 & 1.9988  \\
		\hline
		\multicolumn{5}{c}{by the $P_2-P_2$ WG elements} \\
		\hline
		2     & 1.1821e-01 & --   & 5.4616e-03 & -- \\
		3     & 3.3769e-02 & 1.8075 & 7.6365e-04 & 2.8383 \\
		4     & 8.8861e-03 & 1.9261 & 9.9637e-05 & 2.9382 \\
		5     & 2.2674e-03 & 1.9705 & 1.2670e-05 & 2.9752 \\
		6     & 5.7186e-04 & 1.9873 & 1.5956e-06 & 2.9893 \\
		7     & 1.4354e-04 & 1.9942 & 2.0014e-07 & 2.9951 \\
		\hline
	\end{tabular}
\end{table}

\begin{example}\label{example4.3}(Locking-free test)
	Consider the elasticity problems (\ref{primal_model})-(\ref{bc1}) on square domain $\Omega=(0,1)^2$ with triangular grids in Figure \ref{2Dlevel}. The exact solution $\bu$ is chosen as follows \cite{wangwangliu2023}
	\begin{eqnarray*}
		\bu=\left( \begin{aligned}
			& -(1-\cos(2\pi x))\sin (2 \pi y) \\
			&\quad\, (1-\cos(2\pi y))\sin (2 \pi x) \\
		\end{aligned} \right)
		+\dfrac{1}{\lambda+\mu}\left( \begin{aligned}
			& \sin ( \pi x)\sin ( \pi y) \\
			& \sin ( \pi x)\sin ( \pi y)\\
		\end{aligned} \right),
	\end{eqnarray*}
	where $Lam\acute{e}$ constant $\mu=1$ and
	\begin{eqnarray*}
		\bbf&=&-4  \pi^2\mu \left( \begin{aligned}
			& - \cos(2\pi x)\sin(2\pi y) - \sin(2\pi y)(cos(2\pi x) - 1)\\
			&\quad \,\cos(2\pi y)\sin(2\pi x) +\sin(2\pi x)(cos(2\pi y) - 1)\\
		\end{aligned} \right)\\
		&+&\dfrac{2\pi^2\mu}{\lambda+\mu}\left( \begin{aligned}
			& \sin ( \pi x)\sin ( \pi y) \\
			& \sin ( \pi x)\sin ( \pi y)\\
		\end{aligned} \right)
		-\left( \begin{aligned}
			& \pi^2 \cos(\pi x)\cos(\pi y) - \pi^2\sin(\pi x)\sin(\pi y)\\
			& \pi^2 \cos(\pi x)\cos(\pi y) - \pi^2\sin(\pi x)\sin(\pi y)\\
		\end{aligned} \right),
	\end{eqnarray*}
	\begin{eqnarray*}
		\widehat{\bg}=\left( \begin{aligned}
			& -(1-\cos(2\pi x))\sin (2 \pi y) \\
			&\quad\, (1-\cos(2\pi y))\sin (2 \pi x) \\
		\end{aligned} \right)
		+\dfrac{1}{\lambda+\mu}\left( \begin{aligned}
			& \sin ( \pi x)\sin ( \pi y) \\
			& \sin ( \pi x)\sin ( \pi y)\\
		\end{aligned} \right).
	\end{eqnarray*}
\end{example}

In the computation, we take $\lambda=1$, $\lambda=10^2$, $\lambda=10^4$, and $\lambda=10^6$, respectively. The numerical results by the  $P_1-P_1$ WG elements are provided in Tables \ref{tab:example4.3} - \ref{tab:example4.32}, whereas the numerical outcomes by the $P_2-P_2$ WG elements are illustrated in Figure \ref{exam2_3}.

Tables \ref{tab:example4.3} - \ref{tab:example4.32} demonstrate that optimal convergence rate is achieved in all cases, and the error and the convergence rate are independent of the parameter $\lambda$.
To further illustrate the locking-free property of the WG numerical scheme, we plot the errors in various norms by the $P_2-P_2$ WG elements under different parameter values $\lambda$ in Figure \ref{exam2_3}. Our numerical results reveal that there are no significant differences among the displacement errors in various norms under different parameter values $\lambda$, showing that the WG scheme remains unaffected by the parameter $\lambda$, thus validating its locking-free characteristic.

\begin{table}[H]
	\centering
	\caption{ Error and convergence order of $\bu$ with $\lambda= 1$ by the $P_1-P_1$ WG elements in Example \ref {example4.3}}
	\label{tab:example4.3}
	\begin{tabular}{ccccc}
		\hline
		Level&$\3barQ_h\bu-\bu_h\3bar$&order&$\|Q_0\bu-\bu_0\|$&order\\
		\hline
		2     & 8.4869e+00 & --   & 7.3595e-01 & --\\
		3     & 3.4276e+00 & 1.3080 & 1.6826e-01 & 2.1289 \\
		4     & 1.3982e+00 & 1.2937 & 3.9187e-02 & 2.1022 \\
		5     & 6.3134e-01 & 1.1471 & 9.5265e-03 & 2.0404 \\
		6     & 3.0518e-01 & 1.0488 & 2.3620e-03 & 2.0120 \\
		7     & 1.5117e-01 & 1.0135 & 5.8919e-04 & 2.0032 \\
		\hline
	\end{tabular}
\end{table}

\begin{table}[H]
	\centering
	\caption{Error and convergence order of $\bu$ with $\lambda= 10^2$ by the $P_1-P_1$ WG elements in Example \ref {example4.3}}
	\label{tab:example4.311}
	\begin{tabular}{ccccc}
		\hline
		Level&$\3barQ_h\bu-\bu_h\3bar$&order&$\|Q_0\bu-\bu_0\|$&order\\
		\hline
		2     & 8.5356e+00 &  --   & 7.2393e-01 & -- \\
		3     & 3.4411e+00 & 1.3106 & 1.6342e-01 & 2.1472 \\
		4     & 1.3952e+00 & 1.3024 & 3.7740e-02 & 2.1145 \\
		5     & 6.2668e-01 & 1.1547 & 9.1410e-03 & 2.0457 \\
		6     & 3.0226e-01 & 1.0519 & 2.2638e-03 & 2.0136 \\
		7     & 1.4963e-01 & 1.0144 & 5.6451e-04 & 2.0036 \\
		\hline
	\end{tabular}
\end{table}

\begin{table}[H]
	\centering
	\caption{Error and convergence order of $\bu$ with $\lambda= 10^4$ by the $P_1-P_1$ WG elements in Example \ref {example4.3}}
	\label{tab:example4.312}
	\begin{tabular}{ccccc}
		\hline
		Level&$\3barQ_h\bu-\bu_h\3bar$&order&$\|Q_0\bu-\bu_0\|$&order\\
		\hline
		2     & 8.5375e+00 & --    & 7.2370e-01 & -- \\
		3     & 3.4418e+00 & 1.3107 & 1.6334e-01 & 2.1475 \\
		4     & 1.3954e+00 & 1.3025 & 3.7716e-02 & 2.1146 \\
		5     & 6.2671e-01 & 1.1548 & 9.1348e-03 & 2.0457 \\
		6     & 3.0227e-01 & 1.0520 & 2.2622e-03 & 2.0137 \\
		7     & 1.4963e-01 & 1.0144 & 5.6413e-04 & 2.0036 \\
		\hline
	\end{tabular}
\end{table}

\begin{table}[H]
	\centering
	\caption{Error and convergence order of $\bu$ with $\lambda= 10^6$ by the $P_1-P_1$ WG elements in Example \ref {example4.3}}
	\label{tab:example4.32}
	\begin{tabular}{ccccc}
		\hline
		Level&$\3barQ_h\bu-\bu_h\3bar$&order&$\|Q_0\bu-\bu_0\|$&order\\
		\hline
		2     & 8.5375e+00 & --  & 7.2370e-01 & -- \\
		3     & 3.4418e+00 & 1.3107 & 1.6334e-01 & 2.1475 \\
		4     & 1.3954e+00 & 1.3025 & 3.7716e-02 & 2.1146 \\
		5     & 6.2671e-01 & 1.1548 & 9.1348e-03 & 2.0457 \\
		6     & 3.0227e-01 & 1.0520 & 2.2622e-03 & 2.0137 \\
		7     & 1.4963e-01 & 1.0144 & 5.6402e-04 & 2.0039 \\
		\hline
	\end{tabular}
\end{table}

\begin{figure}[h!]
	\centering
	\includegraphics[width=1 \columnwidth,height=0.6 \linewidth]{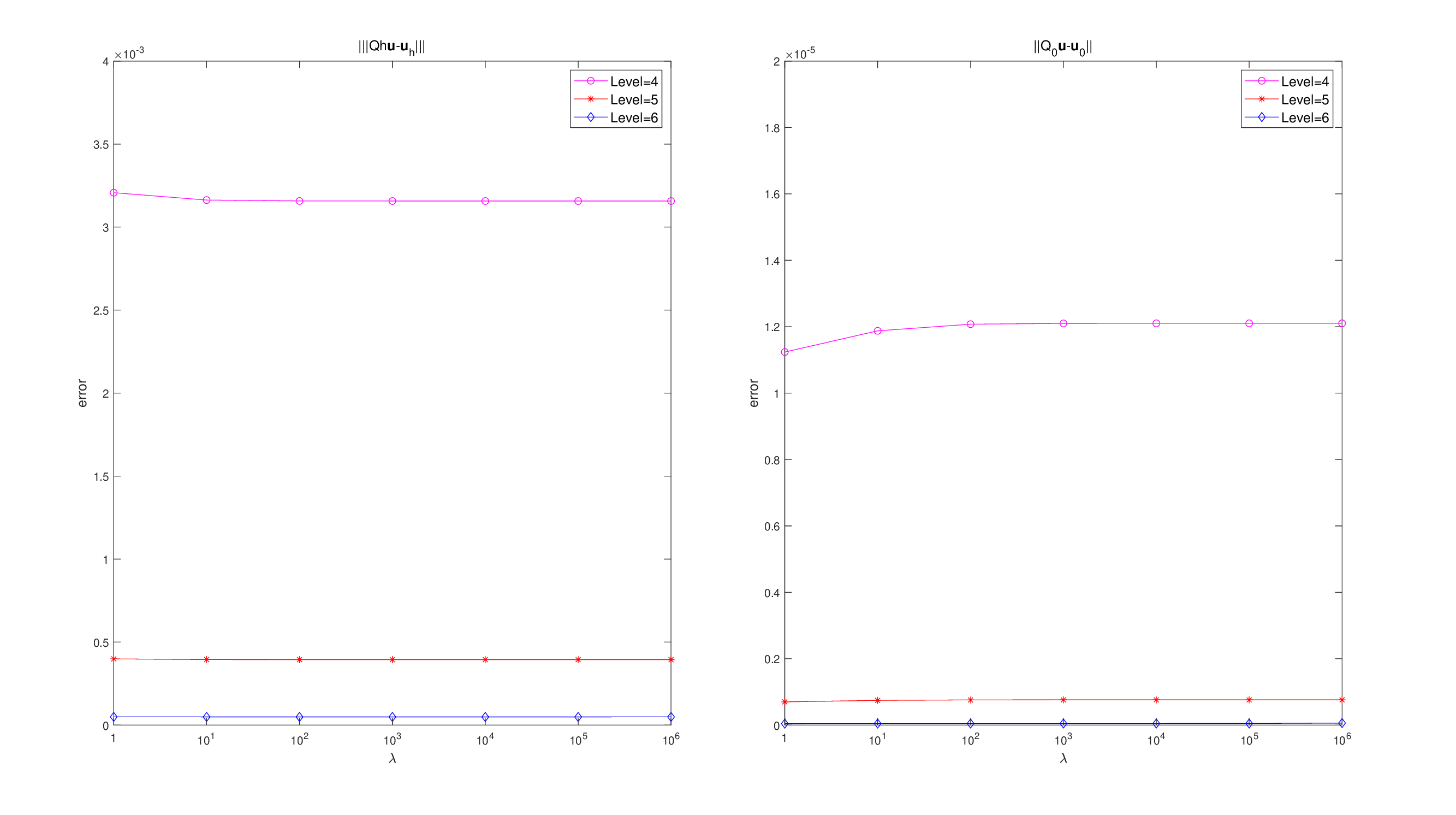}
	\caption{The error of displacement approximation by the $P_2-P_2$ WG elements under different parameter values $\lambda$ from Example \ref{example4.3}. Left: The displacement error in $\3bar\cdot\3bar$-norm. Right: The displacement error in $L^2$-norm}
	\label{exam2_3}
\end{figure}

\begin{example}\label{example4.4}(Locking-free test with unbounded $\lambda \|\nabla\cdot\bu\|_k$)
	To better demonstrate the robustness of the  WG Algorithm \ref{algo-primal},  we consider the case of  $\lambda \|\nabla\cdot\bu\|_k$ unbounded in the following example. Besides, we introduce the standard WG scheme as Algorithm \ref{algo-primal_old} for comparison.
	\begin{algorithm}
		\caption{Standard WG Algorithm}
		\label{algo-primal_old}
		A standard	WG scheme for weak formulation (\ref{model-weak}) is given by finding $\bu_h=\{\bu_0, \bu_b\}\in V_h$ with $\bu_b
		= Q_b \widehat{\bg}$ on $\Gamma $ such that 
		\begin{eqnarray}\label{WGA_primal_old}
			\mathcal{A}^s_h(\bu_h,\bv)=(\bbf,\bv_0), \qquad\forall\, \bv=\{\bv_0, \bv_b\}\in
			V_h^0.
		\end{eqnarray}
	\end{algorithm}
	
	In this example, we consider the elasticity problems (\ref{primal_model})-(\ref{bc1}) on square domain $\Omega=(0,1)^2$ with triangular grids in Figure \ref{2Dlevel}. The exact solution $\bu$ is chosen as follows
	\begin{eqnarray*}
		\bu=\left( \begin{aligned}
			& \sin(\pi x)\sin(\pi y) \\
			& \sin(\pi x)\sin(\pi y) \\
		\end{aligned} \right),
	\end{eqnarray*}
	where $Lam\acute{e}$ constant $\mu=1$ and
	\begin{align*}
		\bbf=&-\mu \left( \begin{aligned}
			& -2\pi^2\sin(\pi x)\sin(\pi y)\\
			&-2\pi^2\sin(\pi x)\sin(\pi y)\\
		\end{aligned} \right)\\
		&-(\lambda+\mu)\left( \begin{aligned}
			& \pi^2 \cos(\pi x)\cos(\pi y) - \pi^2\sin(\pi x)\sin(\pi y) \\
			& \pi^2 \cos(\pi x)\cos(\pi y) - \pi^2\sin(\pi x)\sin(\pi y)\\
		\end{aligned} \right),
	\end{align*}
	\begin{eqnarray*}
		\widehat{\bg}=\left( \begin{aligned}
			& \sin(\pi x)\sin(\pi y) \\
			& \sin(\pi x)\sin(\pi y) \\
		\end{aligned} \right).
	\end{eqnarray*}
\end{example}

\begin{table}[!ht]
	\centering
	\caption{Error and convergence order of displacement $\bu$ in Example \ref {example4.4} by the $P_1-P_1$ WG elements}
	\label{tab:example4.4D1}
	\scalebox{0.90}{
		\begin{tabular}{ccccccccc}
			\hline
			& \multicolumn{4}{c}{Robust WG Algorithm \ref{algo-primal}} & \multicolumn{4}{c}{Standard WG Algorithm \ref{algo-primal_old}} \\
			\hline
			Level&$\3barQ_h\bu-\bu_h\3bar$&order&$\|Q_0\bu-\bu_0\|$&order&$\3barQ_h\bu-\bu_h\3bar$&order&$\|Q_0\bu-\bu_0\|$&order\\
			\hline
			\multicolumn{9}{c}{$\lambda=1$}\\
			\hline
			2     & 1.0388e+00 &  --  & 6.7609e-02 &  --   & 1.9489e+00 &  --   & 1.3797e-01 & --\\
			3     & 4.8696e-01 & 1.0930 & 1.6014e-02 & 2.0779 & 6.6479e-01 & 1.5517 & 2.3331e-02 & 2.5641 \\
			4     & 2.3457e-01 & 1.0538 & 3.9062e-03 & 2.0355 & 2.6238e-01 & 1.3412 & 4.4972e-03 & 2.3752 \\
			5     & 1.1580e-01 & 1.0185 & 9.6880e-04 & 2.0115 & 1.1958e-01 & 1.1336 & 1.0097e-03 & 2.1552 \\
			6     & 5.7689e-02 & 1.0052 & 2.4166e-04 & 2.0032 & 5.8176e-02 & 1.0395 & 2.4431e-04 & 2.0471 \\
			\hline
			\multicolumn{9}{c}{$\lambda=10^2$}\\
			\hline
			2     & 1.0121e+00 &  --  & 4.0674e-02 &   -- & 6.1276e+01 &  --   & 4.2550e+00 & --\\
			3     & 4.8052e-01 & 1.0748 & 7.4022e-03 & 2.4581 & 1.6791e+01 & 1.8676 & 5.8763e-01 & 2.8562 \\
			4     & 2.3221e-01 & 1.0491 & 1.6058e-03 & 2.2046 & 4.3478e+00 & 1.9494 & 7.6391e-02 & 2.9434 \\
			5     & 1.1470e-01 & 1.0176 & 3.8405e-04 & 2.0639 & 1.1060e+00 & 1.9750 & 9.7094e-03 & 2.9760 \\
			6     & 5.7151e-02 & 1.0050 & 9.4883e-05 & 2.0171 & 2.8240e-01 & 1.9695 & 1.2250e-03 & 2.9866 \\
			\hline
			\multicolumn{9}{c}{$\lambda=10^4$}\\
			\hline
			2     & 1.0121e+00 &  --  & 4.0387e-02 &  --  & 6.0245e+03 &  --   & 4.1843e+02 & -- \\
			3     & 4.8052e-01 & 1.0747 & 7.3213e-03 & 2.4637 & 1.6499e+03 & 1.8685 & 5.7752e+01 & 2.8570 \\
			4     & 2.3221e-01 & 1.0491 & 1.5889e-03 & 2.2040 & 4.2666e+02 & 1.9512 & 7.5044e+00 & 2.9441 \\
			5     & 1.1470e-01 & 1.0176 & 3.8029e-04 & 2.0629 & 1.0808e+02 & 1.9809 & 9.5315e-01 & 2.9770 \\
			6     & 5.7151e-02 & 1.0050 & 9.3978e-05 & 2.0167 & 2.7172e+01 & 1.9919 & 1.1998e-01 & 2.9899 \\
			\hline
			\multicolumn{9}{c}{$\lambda=10^6$}\\
			\hline
			2     & 1.0121e+00 &  --   & 4.0384e-02 & --   & 6.0235e+05 & --   & 4.1836e+04 & -- \\
			3     & 4.8052e-01 & 1.0747 & 7.3206e-03 & 2.4638 & 1.6496e+05 & 1.8685 & 5.7742e+03 & 2.8570 \\
			4     & 2.3221e-01 & 1.0491 & 1.5888e-03 & 2.2040 & 4.2658e+04 & 1.9512 & 7.5030e+02 & 2.9441 \\
			5     & 1.1470e-01 & 1.0176 & 3.8026e-04 & 2.0629 & 1.0806e+04 & 1.9809 & 9.5298e+01 & 2.9770 \\
			6     & 5.7151e-02 & 1.0050 & 9.3970e-05 & 2.0167 & 2.7167e+03 & 1.9919 & 1.1996e+01 & 2.9899 \\
			\hline
		\end{tabular}
	}
\end{table}

\begin{table}[!ht]
	\centering
	\caption{Error and convergence order of displacement $\bu$ in Example \ref {example4.4} by the $P_2-P_2$ WG elements}
	\label{tab:example4.4D2}
	\scalebox{0.90}{
		\begin{tabular}{ccccccccc}
			\hline
			& \multicolumn{4}{c}{Robust WG Algorithm \ref{algo-primal}} & \multicolumn{4}{c}{Standard WG Algorithm \ref{algo-primal_old}} \\
			\hline
			Level&$\3bar\bu_h-Q_h\bu\3bar$&order&$\|\bu_0-Q_0\bu\|$&order&$\3bar\bu_h-Q_h\bu\3bar$&order&$\|\bu_0-Q_0\bu\|$&order\\
			\hline
			\multicolumn{9}{c}{$\lambda=1$}\\
			\hline
			2     & 1.7563e-01 &  --  & 1.1594e-02 &  --  & 2.2323e-01 &  --  & 1.6555e-02 & -- \\
			3     & 4.6492e-02 & 1.9175 & 1.4888e-03 & 2.9612 & 4.9644e-02 & 2.1688 & 1.6706e-03 & 3.3088 \\
			4     & 1.1873e-02 & 1.9694 & 1.8790e-04 & 2.9862 & 1.2070e-02 & 2.0401 & 1.9386e-04 & 3.1073 \\
			5     & 2.9928e-03 & 1.9881 & 2.3573e-05 & 2.9947 & 3.0051e-03 & 2.0060 & 2.3762e-05 & 3.0283 \\
			6     & 7.5080e-04 & 1.9950 & 2.9511e-06 & 2.9978 & 7.5157e-04 & 1.9994 & 2.9570e-06 & 3.0064 \\
			\hline
			\multicolumn{9}{c}{$\lambda=10^2$}\\
			\hline
			2     & 1.7665e-01 &   --  & 1.1347e-02 &   --  & 5.4862e+00 &  --   & 4.3735e-01 & -- \\
			3     & 4.7125e-02 & 1.9064 & 1.4510e-03 & 2.9671 & 7.1692e-01 & 2.9359 & 2.8535e-02 & 3.9380 \\
			4     & 1.2061e-02 & 1.9662 & 1.8283e-04 & 2.9885 & 9.1586e-02 & 2.9686 & 1.8167e-03 & 3.9733 \\
			5     & 3.0425e-03 & 1.9870 & 2.2925e-05 & 2.9955 & 1.1818e-02 & 2.9541 & 1.1592e-04 & 3.9701 \\
			6     & 7.6349e-04 & 1.9946 & 2.8695e-06 & 2.9981 & 1.6224e-03 & 2.8648 & 7.6768e-06 & 3.9165 \\
			\hline
			\multicolumn{9}{c}{$\lambda=10^4$}\\
			\hline
			2     & 1.7670e-01 &  --   & 1.1347e-02 &   --  & 5.3982e+02 &   --  & 4.2981e+01 & -- \\
			3     & 4.7144e-02 & 1.9061 & 1.4506e-03 & 2.9676 & 7.0482e+01 & 2.9371 & 2.8029e+00 & 3.9387 \\
			4     & 1.2066e-02 & 1.9661 & 1.8275e-04 & 2.9887 & 8.9481e+00 & 2.9776 & 1.7782e-01 & 3.9784 \\
			5     & 3.0438e-03 & 1.9870 & 2.2914e-05 & 2.9956 & 1.1257e+00 & 2.9907 & 1.1180e-02 & 3.9914 \\
			6     & 7.6382e-04 & 1.9946 & 2.8680e-06 & 2.9981 & 1.4113e-01 & 2.9958 & 7.0060e-04 & 3.9962 \\
			\hline
			\multicolumn{9}{c}{$\lambda=10^6$}\\
			\hline
			2     & 1.7670e-01 &  -- & 1.1347e-02 &  --  & 5.3973e+04 &  --   & 4.2974e+03 & -- \\
			3     & 4.7144e-02 & 1.9061 & 1.4506e-03 & 2.9676 & 7.0471e+03 & 2.9371 & 2.8025e+02 & 3.9387 \\
			4     & 1.2066e-02 & 1.9661 & 1.8275e-04 & 2.9887 & 8.9468e+02 & 2.9776 & 1.7779e+01 & 3.9784 \\
			5     & 3.0439e-03 & 1.9870 & 2.2914e-05 & 2.9955 & 1.1256e+02 & 2.9907 & 1.1178e+00 & 3.9914 \\
			6     & 7.6387e-04 & 1.9945 & 2.8687e-06 & 2.9978 & 1.4110e+01 & 2.9958 & 7.0048e-02 & 3.9962 \\
			\hline
		\end{tabular}
	}
\end{table}

In the computation, we repeat the computational procedure for various parameter values with $\lambda=1,10^2,10^4,10^6$.
The obtained numerical results are presented in Tables \ref{tab:example4.4D1} -\ref{tab:example4.4D2}. In the case of $P_1-P_1$ and $P_2-P_2$ WG elements, it is observed that the displacement errors resulting from the standard WG Algorithm \ref{algo-primal_old} exhibit an increasing trend as the parameter $\lambda$ is incremented. However, the WG Algorithm \ref{algo-primal} shows a distinctive characteristic: the displacement errors remain unaffected by variations in the $Lam\acute{e}$ constant $\lambda$. This finding signifies the robustness of the WG Algorithm \ref{algo-primal} with respect to the $Lam\acute{e}$ constant $\lambda$.

\section*{Acknowledgments}
Y. Wang was supported by the the National Natural Science Foundation of China (grant No. 12171244). R. Zhang and R. Wang were supported by the National Natural Science Foundation of China (grant No. 11971198, 12001230), the National Key Research and Development Program of China (grant No. 2020YFA0713602), and the Key Laboratory of Symbolic Computation and Knowledge Engineering of Ministry of Education of China

\bibliographystyle{siam}
\bibliography{lib}
\end{document}